\documentclass[12pt,a4paper]{amsart}
\usepackage{amsmath,amsthm,amssymb,latexsym,xy,amssymb,enumerate}
\usepackage{enumerate}
\xyoption{all} 

\newtheorem{theorem}{Theorem}
\newtheorem{lemma}[theorem]{Lemma}
\newtheorem{corollary}[theorem]{Corollary}
\newtheorem{proposition}[theorem]{Proposition}

\newtheorem{remark}[theorem]{Remark}
\newtheorem{example}[theorem]{Example}
\newtheorem{question}[theorem]{Question}
\newcommand{\tto}{\twoheadrightarrow}

\begin{document}

\title{On Kiselman quotients of $0$-Hecke monoids}
\author{Olexandr Ganyushkin and Volodymyr Mazorchuk}
\date{\today}
\begin{abstract}
Combining the definition of $0$-Hecke monoids with that of  
Kiselman semigroups, we define what we call Kiselman quotients of 
$0$-Hecke monoids associated with simply laced Dynkin diagrams.
We classify these monoids up to isomorphism, determine their
idempotents and show that they are $\mathcal{J}$-trivial. 
For type $A$ we show that Catalan numbers appear as the 
maximal cardinality of our monoids, in which case the
corresponding monoid is isomorphic to the monoid of all
order-preserving and order-decreasing total transformations
on a finite chain. We construct various representations of these 
monoids by matrices, total transformations and binary relations.
Motivated by these results, with a mixed graph we associate a monoid, 
which  we call a Hecke-Kiselman monoid, and classify such monoids
up to isomorphism. Both Kiselman semigroups and Kiselman quotients 
of $0$-Hecke monoids are natural examples of Hecke-Kiselman monoids.
\end{abstract}

\maketitle

\section{Definitions and description of the results}\label{s1}

Let $\Gamma$ be a simply laced Dynkin diagram (or a disjoint union
of simply laced Dynkin diagrams). Then the
{\em $0$-Hecke monoid} $\mathcal{H}_{\Gamma}$ associated with $\Gamma$
is the monoid generated by idempotents $\varepsilon_i$, where $i$ runs through
the set $\Gamma_0$ of all vertexes of $\Gamma$, subject to the usual braid
relations, namely, $\varepsilon_i\varepsilon_j=\varepsilon_j\varepsilon_i$ 
in the case when $i$ and $j$ are not 
connected in $\Gamma$, and 
$\varepsilon_i\varepsilon_j\varepsilon_i=
\varepsilon_j\varepsilon_i\varepsilon_j$ in the case when $i$ and 
$j$ are connected in $\Gamma$ (see e.g. \cite{NT}). 
Elements of $\mathcal{H}_{\Gamma}$
are in a natural bijection with elements of the Weyl group
$W_{\Gamma}$ of $\Gamma$. The latter follows e.g. from
\cite[Theorem~1.13]{Ma} as the semigroup algebra of
the monoid $\mathcal{H}_{\Gamma}$ is canonically isomorphic to 
the specialization of the Hecke algebra $\mathcal{H}_q(W_{\Gamma})$
at $q=0$, which also explains the name. This specialization was studied by
several authors, see \cite{No,Ca,McN,Fa,HNT,NT2} and references
therein. The monoid $\mathcal{H}_{\Gamma}$ appears for example in
\cite{FG,HST,HST2}. One has to note that $\mathcal{H}_{\Gamma}$
appears in articles where the emphasis is made on its semigroup algebra 
and not its structure as a monoid. Therefore semigroup properties of
$\mathcal{H}_{\Gamma}$ are not really spelled out in the above papers.
However, with some efforts one can derive from the above literature 
that the monoid $\mathcal{H}_{\Gamma}$ is $\mathcal{J}$-trivial
(we will show this in Subsection~\ref{s2.1}) and has $2^n$ idempotents, 
where $n$ is the number of vertexes in $\Gamma$ (we will show this 
in Subsection~\ref{s2.2}).

Another example of an idempotent generated $\mathcal{J}$-trivial monoid
with $2^n$ idempotents (where $n$ is the number of generators) is 
{\em Kiselman's semigroup} $\mathbf{K}_n$, defined as follows: it is generated
by idempotents $e_i$, $i=1,2,\dots,n$, subject to the relations
$e_ie_je_i=e_je_ie_j=e_ie_j$ for all $i>j$ (see \cite{Go}). 
This semigroup was studied in \cite{KM,Al}. In particular, in \cite{KM}
it was shown that $\mathbf{K}_n$ has a faithful representation by $n\times n$
matrices with non-negative integer coefficients.

The primary aim of this paper is to study natural mixtures of these two
semigroups, which we call {\em Kiselman quotients} of $\mathcal{H}_{\Gamma}$.
These are defined as follows: choose any 
orientation $\vec{\Gamma}$ of $\Gamma$ and define the
semigroup $\mathbf{K}\mathcal{H}_{\vec{\Gamma}}$ as the quotient 
of $\mathcal{H}_{\Gamma}$ obtained by imposing the additional relations
$\varepsilon_i\varepsilon_j\varepsilon_i=
\varepsilon_j\varepsilon_i\varepsilon_j=\varepsilon_i\varepsilon_j$ 
in all cases when 
$\vec{\Gamma}$ contains the arrow $\xymatrix{i\ar[r]&j}$.
These relations are natural combinations of the relations defining
$\mathcal{H}_{\Gamma}$ and $\mathbf{K}_{n}$. Our first result is the
following theorem:

\begin{theorem}\label{thm1}
\begin{enumerate}[$(i)$]
\item\label{thm1.1} The semigroup 
$\mathbf{K}\mathcal{H}_{\vec{\Gamma}}$ is $\mathcal{J}$-trivial.
\item\label{thm1.2} 
The set $\mathbf{E}:=\{\varepsilon_i:i\in \Gamma_0\}$ is the unique 
irreducible generating system for 
$\mathbf{K}\mathcal{H}_{\vec{\Gamma}}$.
\item\label{thm1.3} The semigroup 
$\mathbf{K}\mathcal{H}_{\vec{\Gamma}}$ contains $2^n$ idempotents, 
where $n$ is the number of vertexes in $\Gamma$.
\item\label{thm1.4} The semigroups 
$\mathbf{K}\mathcal{H}_{\vec{\Gamma}}$ and
$\mathbf{K}\mathcal{H}_{\vec{\Lambda}}$ are isomorphic if and only
if the directed graphs $\Gamma$ and $\Lambda$ are isomorphic.
\item\label{thm1.5} The semigroups 
$\mathbf{K}\mathcal{H}_{\vec{\Gamma}}$ and
$\mathbf{K}\mathcal{H}_{\vec{\Lambda}}$ are anti-isomorphic if 
and only if the directed graphs $\Gamma$ and $\Lambda$ are anti-isomorphic.
\item\label{thm1.6} If $\Gamma$ is a Dynkin diagram of type $A_n$, then
$|\mathbf{K}\mathcal{H}_{\vec{\Gamma}}|
\leq C_{n+1}$, where $C_n:=\frac{1}{n+1}\binom{2n}{n}$ is the 
$n$-th Catalan number.
\item\label{thm1.7} If $\Gamma$ is a Dynkin diagram of type $A_n$, then
$|\mathbf{K}\mathcal{H}_{\vec{\Gamma}}|=C_{n+1}$ if and only if
$\vec{\Gamma}$ is isomorphic to the graph
\begin{displaymath}
\xymatrix{
\bullet\ar[r]&\bullet\ar[r]&\bullet\ar[r]&\dots\ar[r]&
\bullet\ar[r]&\bullet
}.
\end{displaymath}
\item\label{thm1.8} If $\vec{\Gamma}$ is as in 
$(\ref{thm1.7})$, then the semigroup 
$\mathbf{K}\mathcal{H}_{\vec{\Gamma}}$ is isomorphic to the
semigroup $\mathcal{C}_{n+1}$ of all order-preserving and order-decreasing
total transformations of $\{1,2,\dots,n,n+1\}$ $($see 
\cite[Chapter~14]{GM}$)$.
\end{enumerate}
\end{theorem}

The semigroup $\mathcal{C}_{n+1}$ appears in various disguises in 
\cite{So,Pi,HT,GM}. Its presentation can be derived from \cite{So}, 
however, in the present paper this semigroup appears in a different 
context and our proof is much less technical.
In \cite{So} it is also observed that the cardinality
of the semigroup with this presentation is given by Catalan numbers.
Classically, Catalan numbers appear in semigroup theory as the cardinality 
of the so-called Temperley-Lieb semigroup $\mathfrak{TL}_n$,
see \cite[6.25(g)]{St1}. That Catalan numbers appear as the cardinality 
of $\mathcal{C}_{n+1}$ was first observed in \cite{Hi} (with an 
unnecessarily difficult proof, see \cite[6.19(u)]{St1} for a 
straightforward argument). In \cite{GM0} it was shown that Catalan numbers 
also appear as the maximal cardinality of a nilpotent subsemigroup in 
the semigroup $\mathcal{IO}_n$ of all partial order-preserving injections 
on $\{1,2,\dots,n\}$ (see also \cite{GM01} for an alternative argument).

Motivated by both Kiselman semigroups and Kiselman quotients of
$0$-Hecke monoids, we propose the notion of Hecke-Kiselman semigroups
associated with an arbitrary mixed (finite) graph. A mixed graph
is a simple graph in which edges can be both oriented and unoriented.
Such graph is naturally given by an anti-reflexive binary relation 
$\Theta$ on a finite set (see Subsection~\ref{s5.1}). The corresponding
Hecke-Kiselman semigroup $\mathbf{HK}_{\Theta}$ is generated by 
idempotents $e_i$ indexed by vertexes of the graph, subject to the
following relations:
\begin{itemize}
\item if $i$ and $j$ are not connected by any edge, then $e_ie_j=e_je_i$;
\item if $i$ and $j$ are connected by an unoriented edge, then
$e_ie_je_i=e_je_ie_j$;
\item if $i$ and $j$ are connected by an oriented edge
$i\to j$, then $e_ie_je_i=e_je_ie_j=e_ie_j$.
\end{itemize}
Our second result is:

\begin{theorem}\label{thmn1}
Let $\Theta$ and $\Phi$ be two anti-reflexive binary relations
on finite sets. Then $\mathbf{HK}_{\Theta}\cong\mathbf{HK}_{\Phi}$
if and only if the corresponding mixed graphs are isomorphic.
\end{theorem}

The paper is organized as follows: Theorem~\ref{thm1} is proved in
Section~\ref{s2}. In Section~\ref{s3} we construct 
representations of $\mathbf{K}\mathcal{H}_{\vec{\Gamma}}$ by total 
transformations, matrices with non-negative integral coefficients and 
binary relations. We also describe simple and  indecomposable projective 
linear representations of $\mathbf{K}\mathcal{H}_{\vec{\Gamma}}$ over any field. In Section~\ref{s4} we give an application of our results
to combinatorial interpretations of Catalan numbers.
Finally, in Section~\ref{s5} we present a general definition of
Hecke-Kiselman semigroups and prove Theorem~\ref{thmn1}. As a corollary,
we obtain a formula for the number of isomorphism classes of
Hecke-Kiselman semigroups on a given set. We complete the paper
with a short list of open problems on Hecke-Kiselman semigroups.
\vspace{5mm}

\noindent
{\bf Acknowledgments.} The paper was written during the visit of the
first author to Uppsala University. The financial support and 
hospitality of Uppsala University are gratefully acknowledged.
For the second author the research was partially supported by the
Swedish Research Council. We thank Ganna Kudryavtseva for her comments.

\section{Proof of Theorem~\ref{thm1}}\label{s2}

As usual, we denote by $\Gamma_0$ the set of vertexes of the graph 
$\Gamma$ and set $n=|\Gamma_0|$. Consider the free monoid 
$\mathfrak{W}_n$ generated by $a_1,\dots,a_n$ and the canonical
epimorphism $\varphi:\mathfrak{W}_n\to 
\mathbf{K}\mathcal{H}_{\vec{\Gamma}}$, defined by 
$\varphi(a_i)=\varepsilon_i$, $i\in \Gamma_0$. We identify 
$\mathbf{K}\mathcal{H}_{\vec{\Gamma}}$ with the quotient of
$\mathfrak{W}_n$ by $\mathrm{Ker}(\varphi)$.

For $w\in \mathfrak{W}_n$ the {\em content} $\mathfrak{c}(w)$ is defined
as the set of indexes for which the corresponding generators appear 
in $w$. For any relation $v=w$ used in the definition of 
$\mathbf{K}\mathcal{H}_{\vec{\Gamma}}$ we have
$\mathfrak{c}(v)=\mathfrak{c}(w)$. This implies that 
for any $\alpha\in \mathbf{K}\mathcal{H}_{\vec{\Gamma}}$
(which we interpret as an equivalence class in 
$\mathrm{Ker}(\varphi)$) and any $v,w\in \alpha$ we have
$\mathfrak{c}(v)=\mathfrak{c}(w)$. Hence we may define
$\mathfrak{c}(\alpha)$ as $\mathfrak{c}(v)$ for any $v\in\alpha$.

\subsection{Proof of statement \eqref{thm1.1}}\label{s2.1}
We start with the following statement, which we could not find any
explicit reference to.

\begin{lemma}\label{lem2}
The monoid $\mathcal{H}_{\Gamma}$ is $\mathcal{J}$-trivial.
\end{lemma}

\begin{proof}
For $w\in W_{\Gamma}$ denote by $H_w\in \mathcal{H}_{\Gamma}$ the corresponding
element (if $w=s_{i_1}s_{i_2}\cdots s_{i_k}$ is a reduced decomposition of
$w$ into a product of simple reflections, then 
$H_w=\varepsilon_{i_1}\varepsilon_{i_2}\cdots \varepsilon_{i_k}$). 
Let $l:W_{\Gamma}\to\{0,1,\dots\}$ denote the classical length
function. Then the usual multiplication properties of the Hecke algebra 
(\cite[Lemma~1.12]{Ma}) read as follows:
\begin{equation}\label{eq1}
\varepsilon_i H_w=
\begin{cases}
H_{s_iw} & l(s_iw)>l(w);\\
H_w&\text{ otherwise};
\end{cases}
\quad
H_w\varepsilon_i =
\begin{cases}
H_{ws_i} & l(ws_i)>l(w);\\
H_w&\text{ otherwise}.
\end{cases}
\end{equation}

Hence for any $w\in W_{\Gamma}$ the two-sided ideal 
$\mathcal{H}_{\Gamma}H_w\mathcal{H}_{\Gamma}$ consists of $H_w$
and, possibly, some elements of strictly bigger length. In particular,
for any $x\in \mathcal{H}_{\Gamma}H_w\mathcal{H}_{\Gamma}$ such that
$x\neq H_w$ we have $\mathcal{H}_{\Gamma}H_w\mathcal{H}_{\Gamma}\neq \mathcal{H}_{\Gamma}x\mathcal{H}_{\Gamma}$. The claim follows.
\end{proof}

As any quotient of a finite $\mathcal{J}$-trivial semigroup is 
$\mathcal{J}$-trivial (see e.g. \cite[Chapter~VI, Section~5]{Ll}),
statement \eqref{thm1.1} follows from Lemma~\ref{lem2}.

\subsection{Proof of statement \eqref{thm1.2}}\label{s2.2}

The set $\mathbf{E}$ generates $\mathbf{K}\mathcal{H}_{\vec{\Gamma}}$
by definition. We claim that this generating system is irreducible.
Indeed, if we can write $e_i$ as a product $w$ of generators,
then $\mathfrak{c}(w)=\{i\}$, implying $w=e_i$. Hence 
$\mathbf{E}$ is irreducible. Further,
we know that $\mathbf{K}\mathcal{H}_{\vec{\Gamma}}$
is $\mathcal{J}$-trivial from statement \eqref{thm1.1}.
Uniqueness of the irreducible generating system in a 
$\mathcal{J}$-trivial monoid was established in \cite[Theorem~2]{Do}.
This implies statement \eqref{thm1.2}.

\subsection{Proof of statement \eqref{thm1.3}}\label{s2.3}
Identify $\Gamma_0$ with $\{1,2,\dots,n\}$ such that 
$\xymatrix{i\ar[r]&j}$ implies $i>j$ for all $i$ and $j$.
Then the mapping $e_i\mapsto \varepsilon_i$, $i\in \{1,2,\dots,n\}$,
extends to an epimorphism $\psi:\mathbf{K}_n\tto
\mathbf{K}\mathcal{H}_{\vec{\Gamma}}$ (as all relations for
generators of $\mathbf{K}_n$ are satisfied by the corresponding
generators of $\mathbf{K}\mathcal{H}_{\vec{\Gamma}}$).

By \cite{KM}, the semigroup $\mathbf{K}_n$ has exactly $2^n$
idempotents, all having different contents. As $\psi$ preserves
the content, we obtain $2^n$ different idempotents in 
$\mathbf{K}\mathcal{H}_{\vec{\Gamma}}$. As any epimorphism
of finite semigroups induces an epimorphism on the corresponding
sets of idempotents, the statement \eqref{thm1.3} follows.

For completeness, we include the following statement which describes 
idempotents in $\mathcal{H}_{\Gamma}$ in terms of longest elements
for parabolic subgroups of $W_{\Gamma}$ (this claim can also be deduced from
\cite[Lemma~2.2]{No}).

\begin{lemma}\label{lem4}
For any $X\subset \Gamma_0$ left $w_X$ denote the longest element 
in the parabolic subgroup of $W_{\Gamma}$ associated with 
$X$ ($w_{\varnothing}=e$).
Then $H_{w_X}\in \mathcal{H}_{\Gamma}$ is an idempotent, and
every idempotent of $\mathcal{H}_{\Gamma}$ has the form 
$H_{w_X}$ for some $X$ as above. In particular, 
$\mathcal{H}_{\Gamma}$ has $2^n$ idempotents.
\end{lemma}

\begin{proof}
Let $w\in W_{\Gamma}$. Assume that $H_w$ is an idempotent.
From \eqref{eq1} it follows that $H_wH_w=H_w$ implies that 
$\varepsilon_i H_w=H_w\varepsilon_i=H_w$ for any $i\in\mathfrak{c}(H_w)$.
In particular, for any $i\in\mathfrak{c}(H_w)$ we have
$l(s_iw)<l(w)$ and $l(ws_i)<l(w)$, in other words, both the left and
the right descent sets of $w$ contain all simple reflections
appearing in any reduced decomposition of $w$. From \cite[2.3]{BB} 
it now follows that $w$ is the longest element of the parabolic subgroup
of $W_{\Gamma}$, generated by all $s_i$, $i\in\mathfrak{c}(H_w)$. 

On the other hand, if $w$ is the longest element from some parabolic subgroup
of $W_{\Gamma}$, then the same arguments imply 
$\varepsilon_i H_w=H_w\varepsilon_i=H_w$ for any $i\in\mathfrak{c}(H_w)$
and hence $H_wH_w=H_w$. The claim follows.
\end{proof}

\subsection{Proof of statement \eqref{thm1.4}}\label{s2.4}

This statement follows from a more general statement of
Theorem~\ref{thm4}, which will be proved in Subsection~\ref{s5.3}.

\subsection{Proof of statement \eqref{thm1.5}}\label{s2.5}

By Proposition~\ref{prop7}, which will be proved in a more
general situation in Subsection~\ref{s5.1}, the semigroups
$\mathbf{K}\mathcal{H}_{\vec{\Gamma}}$ and
$\mathbf{K}\mathcal{H}_{\vec{\Lambda}}$ are anti-isomorphic
if and only if $\mathbf{K}\mathcal{H}_{\vec{\Gamma}}$
and $\mathbf{K}\mathcal{H}_{\vec{\Lambda}^{\mathrm{op}}}$
are isomorphic. By statement \eqref{thm1.4}, the latter is the case
if and only if $\vec{\Gamma}$ and 
$\vec{\Lambda}^{\mathrm{op}}$ are isomorphic,
which implies statement \eqref{thm1.5}.

\subsection{Proof of statement \eqref{thm1.6}}\label{s2.6}

Since $\Gamma$ is now of type $A_n$, the group $W_{\Gamma}$ is isomorphic 
to the  symmetric group $S_{n+1}$. Consider the canonical projection
$\mathcal{H}_{\Gamma}\tto\mathbf{K}\mathcal{H}_{\vec{\Gamma}}$. Then any
equivalence class of the kernel of this projection contains
some element of minimal possible length (maybe not unique). 
Let $H_w$ be such an element and $w=s_{i_1}s_{i_2}\dots s_{i_k}$ 
be a reduced decomposition in $W_{\Gamma}$. Then this reduced 
decomposition cannot contain any subword of the
form $s_is_{j}s_i$ (where $i$ and $j$ are connected in $\Gamma$), 
in other words, $w$ is a  {\em short-braid avoiding} permutation. 
Indeed, otherwise $H_w$ would be equivalent to $H_{w'}$, where $w'$ is a 
shorter word obtained from $w$ by changing $s_is_{j}s_i$ to either 
$s_is_{j}$ or $s_{j}s_i$ depending on the direction of the arrow
between $i$ and $j$ in $\vec{\Gamma}$, which would contradict 
our choice of  $w$. 

Therefore the cardinality of $\mathbf{K}\mathcal{H}_{\vec{\Gamma}}$
does not exceed the number of short-braid avoiding
elements in $S_{n+1}$. These are known to correspond to 
$321$-avoiding permutations (see e.g. \cite[Theorem~2.1]{BJS}). 
The number of $321$-avoiding permutations in $S_{n+1}$  is known to 
be $C_{n+1}$ (see e.g. \cite[6.19(ee)]{St1}). Statement 
\eqref{thm1.6} follows.

\subsection{Proof of statement \eqref{thm1.7}}\label{s2.7}

Assume first that $\vec{\Gamma}$ coincides with
\begin{equation}\label{eq2}
\xymatrix{ 
1 & 2\ar[l] & 3\ar[l] & \ar[l]\dots &\ar[l] n
}.
\end{equation}
From \eqref{thm1.6} we already know that 
$|\mathbf{K}\mathcal{H}_{\vec{\Gamma}}|\leq C_{n+1}$. 
For $i=1,2,\dots,n$ denote by $T_i$ the following transformation of
$\{1,2,\dots,n+1\}$:
\begin{equation}\label{eq31}
\left(\begin{array}{cccccccccc}
1&2&\dots&i-1&i&i+1&i+2&\dots&n&n+1\\
1&2&\dots&i-1&i&i&i+2&\dots&n&n+1\\
\end{array}\right).
\end{equation}
The semigroup $\mathcal{C}_{n+1}$, generated by the $T_i$'s is the
semigroup of all order-decreasing and order-preserving total transformations
on the set $\{1,2,\dots,n+1\}$, see \cite[Chapter~14]{GM}. One easily checks 
that the $T_i$'s are idempotent, that $T_iT_j=T_jT_i$ if $|i-j|>1$
and that $T_iT_{i+1}T_i=T_{i+1}T_iT_{i+1}=T_{i}T_{i+1}$ for all 
$i=1,2,\dots,n-1$. Therefore, sending $\varepsilon_i$ to
$T_{n+1-i}$ for all $i$ defines an epimorphism from 
$\mathbf{K}\mathcal{H}_{\vec{\Gamma}}$ to $\mathcal{C}_{n+1}$.
As $|\mathcal{C}_{n+1}|=C_{n+1}$ by \cite[6.25(g)]{St1}, we obtain that
$|\mathbf{K}\mathcal{H}_{\vec{\Gamma}}|\geq C_{n+1}$ and hence
$|\mathbf{K}\mathcal{H}_{\vec{\Gamma}}|= C_{n+1}$.

Assume now that $\vec{\Gamma}$ is not isomorphic to
\eqref{eq2}. Then either $\vec{\Gamma}$ or 
$\vec{\Gamma}^{\mathrm{op}}$ must contain the
following full subgraph:
\begin{equation}\label{eq3}
\xymatrix{ 
i\ar[r]&j&\ar[l]k
}.
\end{equation}
Using \eqref{thm1.5} and the fact that 
$|\mathbf{K}\mathcal{H}_{\vec{\Gamma}}|=
|\mathbf{K}\mathcal{H}_{\vec{\Gamma}}^{\mathrm{op}}|$,
without loss of generality we may assume that 
$\vec{\Gamma}$ contains \eqref{eq3}. It is easy to see
that the element $s_js_is_ks_j\in W_{\Gamma}$ is short-braid avoiding. On the 
other hand, because of the arrows 
$\xymatrix{i\ar[r]&j}$ and $\xymatrix{k\ar[r]&j}$ we have
\begin{displaymath}
\varepsilon_j\varepsilon_i\varepsilon_k\varepsilon_j=
\varepsilon_j\varepsilon_i\varepsilon_j\varepsilon_k\varepsilon_j=
\varepsilon_j\varepsilon_i\varepsilon_j\varepsilon_k=
\varepsilon_j\varepsilon_i\varepsilon_k.
\end{displaymath}
Note that $s_js_is_k$ is again short-braid avoiding. It follows that 
in this case some different short-braid avoiding permutations 
correspond to equal elements of $\mathbf{K}\mathcal{H}_{\vec{\Gamma}}$. 
Hence $|\mathbf{K}\mathcal{H}_{\vec{\Gamma}}|$
is strictly smaller than the total number of short-braid avoiding
permutations, implying statement \eqref{thm1.7}.

\subsection{Proof of statement \eqref{thm1.8}}\label{s2.8}
Statement \eqref{thm1.8} follows from the observation that the epimorphism
from $\mathbf{K}\mathcal{H}_{\vec{\Gamma}}$ to $\mathcal{C}_{n+1}$,
constructed in the first part of our proof of statement \eqref{thm1.7},
is in fact an isomorphism as 
$|\mathbf{K}\mathcal{H}_{\vec{\Gamma}}|=|\mathcal{C}_{n+1}|= C_{n+1}$.

\section{Representations of 
$\mathbf{K}\mathcal{H}_{\vec{\Gamma}}$}\label{s3}

In this section $\Gamma$ is a disjoint union of Dynkin diagrams and
$\vec{\Gamma}$ is obtained from $\Gamma$ by orienting all edges in some way.

\subsection{Representations by total transformations}\label{s3.1}

In this subsection we generalize the action described in 
Subsection~\ref{s2.8}. In order to minimize the cardinality of the
set our transformations operate on, we assume that $\vec{\Gamma}$ is such 
that the indegree of the triple point of 
$\vec{\Gamma}$ (if such a point exists)  is at most one. 
This is always satisfied either by $\vec{\Gamma}$ or 
by $\vec{\Gamma}^{\mathrm{op}}$. In type $A$ we have no restrictions.
Using the results of  Subsection~\ref{s2.5}, we thus construct either 
a left or a right action of $\mathbf{K}\mathcal{H}_{\vec{\Gamma}}$ 
for every $\vec{\Gamma}$.

Consider the set $M$ defined as the disjoint union of the 
following sets: the set 
$\vec{\Gamma}_1$ of all edges in $\vec{\Gamma}$, the set 
$\vec{\Gamma}_0^0$ of all sinks in $\vec{\Gamma}$ (i.e. vertexes of
outdegree zero), the set $\vec{\Gamma}_0^1$ of all sinks in 
$\vec{\Gamma}$ of indegree two, and the set 
$\vec{\Gamma}_0^2$ of all sources in $\vec{\Gamma}$ (i.e. vertexes
of indegree zero). Fix some injection $g:\vec{\Gamma}_0^0\cup 
\vec{\Gamma}_0^1
\to \vec{\Gamma}_1$ which maps a vertex to some edge terminating in this
vertex (this is uniquely defined if the indegree of our vertex is
one, but there is a choice involved if this indegree is two).
Note that under our assumptions any vertex which is not a sink 
has indegree at most one.

For $i\in\Gamma_0$ define the total transformation $\tau_i$ of $M$
as follows:
\begin{equation}\label{eq4}
\tau_i(x)=
\begin{cases}
y,& \xymatrix{\ar[r]|-y&i\ar[r]|-x&};\\
i,& \xymatrix{i\ar[r]|-x&} \text{ and $i$ is a source};\\
g(i),& \text{$x=i$ is a sink};\\
x,&\text{ otherwise}.
\end{cases}
\end{equation}

\begin{proposition}\label{prop21}
Formulae $($\ref{eq4}$)$ define a representation of 
$\mathbf{K}\mathcal{H}_{\vec{\Gamma}}$ by total transformations on $M$.
\end{proposition}

\begin{proof}
To prove the claim we have to check that the $\tau_i$'s satisfy
the defining relations for $\mathbf{K}\mathcal{H}_{\vec{\Gamma}}$.
Relations $\tau_i^2=\tau_i$ and  $\tau_i\tau_j=\tau_j\tau_i$ if 
$i$ and $j$ are not connected follow directly from the definitions.
So, we are left to check that
$\tau_i\tau_j\tau_i=\tau_j\tau_i\tau_j=\tau_i\tau_j$ if we have
\begin{displaymath}
\xymatrix{*+[F]{\Lambda}&i\ar[r]\ar@{.}[l]&j\ar@{.}[r]&*+[F]{\Lambda'}}. 
\end{displaymath}
Every point in $M$ coming from $\Lambda$ or $\Lambda'$ is invariant 
under both $\tau_i$ or $\tau_j$, so on such elements the relations are 
obviously satisfied. 

The above reduces checking of our relation to the elements coming
from the following local situations:
\begin{gather*}
\xymatrix{&i\ar[r]\ar[l]&j\ar[r]&},\\
\xymatrix{\ar[r]&i\ar[r]&j&\ar[l]},\\
\xymatrix{\ar[r]&i\ar[r]&j\ar[r]&}.
\end{gather*}
In all these cases all relations are easy to check 
(and the nontrivial  ones reduce to the corresponding 
relations for the representation considered in Subsection~\ref{s2.8}).
This completes the proof.
\end{proof}

\begin{question}\label{qes1}
{\rm  
Is the representation constructed above faithful?
}
\end{question}

If $\vec{\Gamma}$ is given by \eqref{eq2}, then $M$ contains $n+1$
elements and it is easy to see that it is equivalent to the 
representation  considered in Subsection~\ref{s2.8}. In particular, as was
shown there, this representation is faithful. So in this case the
answer to Question~\ref{qes1} is positive.

\subsection{Linear integral representations}\label{s3.2}

Let $V$ denote the free abelian group generated by 
$v_i$, $i\in\Gamma_0$. For $i\in \Gamma_0$ define the homomorphism
$\theta_i$ of $V$ as follows:
\begin{displaymath}
\theta_i(v_j)=
\begin{cases}
v_j, & i\neq j;\\
\displaystyle \sum_{k\to i} v_k, & i=j. 
\end{cases}
\end{displaymath}

\begin{proposition}\label{prop22}
Mapping $\varepsilon_i$ to $\theta_i$ extends uniquely to a
homomorphism from $\mathbf{K}\mathcal{H}_{\vec{\Gamma}}$ to the semigroup
$\mathrm{End}_{\mathbb{Z}}(V)$.
\end{proposition}

\begin{proof}
To prove the claim we have to check that the $\theta_i$'s satisfy
the defining relations for $\mathbf{K}\mathcal{H}_{\vec{\Gamma}}$.
We do this below. 

{\em Relation $\theta_i^2=\theta_i$.} If $j\neq i$, then
$\theta_i^2(v_j)=\theta_i(v_j)=v_j$ by definition. As $\Gamma$
contains no loops, we also have 
\begin{displaymath}
\theta_i^2(v_i)=\theta_i(\sum_{k\to i} v_k)=
\sum_{k\to i}\theta_i(v_k)\overset{k\neq i}{=}\sum_{k\to i} v_k=
\theta_i(v_i).
\end{displaymath}

{\em Relation $\theta_i\theta_j=\theta_j\theta_i$ if
$i$ and $j$ are not connected.} If $k\neq i,j$, then 
$\theta_i\theta_j(v_k)=\theta_j\theta_i(v_k)=v_k$ by definition.
By symmetry, it is left to show that 
$\theta_i\theta_j(v_i)=\theta_j\theta_i(v_i)$. We have
\begin{displaymath}
\theta_i\theta_j(v_i)\overset{j\neq i}{=}
\theta_i(v_i)=\sum_{k\to i} v_k\overset{k\neq j}{=}
\sum_{k\to i} \theta_j(v_k)=
\theta_j(\sum_{k\to i} v_k)=
\theta_j\theta_i(v_i).
\end{displaymath}

{\em Relation $\theta_i\theta_j\theta_i=\theta_j\theta_i\theta_j
=\theta_i\theta_j$ if we have $\xymatrix{i\ar[r]&j}$.} If $k\neq i,j$, 
then $\theta_i\theta_j(v_k)=\theta_j\theta_i(v_k)=v_k$ by definition
and our relation is satisfied. Further we have 
\begin{displaymath}
\theta_i\theta_j\theta_i(v_i)=
\sum_{k\to i} \theta_i\theta_j(v_k)
\overset{k\neq i,j}{=}\sum_{k\to i} v_k
\end{displaymath}
ans similarly both $\theta_j\theta_i\theta_j(v_i)$ and 
$\theta_i\theta_j(v_i)$ equal $\sum_{k\to i} v_k$ as well. Finally, 
we have 
\begin{multline*}
\theta_i\theta_j\theta_i(v_j)\overset{i\neq j}{=}
\theta_i\theta_j(v_j)=\theta_i(
\sum_{k\to j} v_k)=\sum_{k\to j} \theta_i(v_k)=\\
=\theta_i(v_i)+\sum_{k\to j, k\neq i} \theta_i(v_k)=
\sum_{k\to i} v_k+ \sum_{k\to j, k\neq i} v_k.
\end{multline*}
As $\Gamma$ contains no loops, the result is obviously preserved 
by $\theta_j$ giving the desired relation.
This completes the proof.
\end{proof}

The representation given by Proposition~\ref{prop22} is a generalization
of Kiselman's representation for $\mathbf{K}_n$, see \cite[Section~5]{KM}.
Using the canonical anti-involution (transposition) for linear operators
and Subsection~\ref{s2.5}, from the above we also obtain a representation 
for $\mathbf{K}\mathcal{H}_{\vec{\Gamma}}^{\mathrm{op}}$.

\begin{question}\label{qes2}
{\rm  
Is the representation constructed above faithful
(as semigroup representation)?
}
\end{question}

If $\vec{\Gamma}$ is given by \eqref{eq2}, then the linear representation
of $\mathbf{K}\mathcal{H}_{\vec{\Gamma}}$ given by Proposition~\ref{prop22} is
just a linearization of the representation from Subsection~\ref{s3.1}.
Hence from Subsection~\ref{s2.8} it follows that the answer to 
Question~\ref{qes2} is positive in this case.

If we identify linear operators on $V$ with $n\times n$ integral
matrices with respect to the basis $\{v_i:i\in\Gamma_0\}$, we obtain
a representation of $\mathbf{K}\mathcal{H}_{\vec{\Gamma}}$ by $n\times n$ 
matrices with non-negative integral coefficients. Call this
representation $\Theta$. 

\begin{lemma}\label{lem26}
The representation 
$\Theta$ is a representation of $\mathbf{K}\mathcal{H}_{\vec{\Gamma}}$ by
$(0,1)$-matrices (i.e. matrices with coefficients $0$ or $1$).
\end{lemma}

\begin{proof}
For $\alpha\in\mathbf{K}\mathcal{H}_{\vec{\Gamma}}$ we show that $\Theta(\alpha)$
is a $(0,1)$-matrix by induction on the length of $\alpha$ (that is the
length of the shortest decomposition of $\alpha$ into a product of
canonical generators). If $\alpha=\varepsilon$, the claim is obvious.
If $\alpha$ is a generator, the claim follows from the definition
of $\Theta$ (as $\Gamma$ is a simple graph).

Let $\theta_{\alpha}$ denote the homomorphism of $V$ corresponding to 
$\alpha$. To prove the induction step we consider some shortest decomposition
$\alpha=\varepsilon_{i_1}\varepsilon_{i_2}\cdots\varepsilon_{i_p}$
and set $\beta=\varepsilon_{i_1}\varepsilon_{i_2}\cdots\varepsilon_{i_{p-1}}$.
Then for any $j\neq i_p$ we have
$\theta_{\alpha}(v_j)=\theta_{\beta}\theta_{i_p}(v_j)=\theta_{\beta}(v_j)$, 
which is a $(0,1)$-linear combination of the $v_k$'s 
by the inductive assumption.

For $v_{i_p}$ we use induction on $p$ to show that 
$\theta_{\alpha}(v_{i_p})$ is a $(0,1)$-linear combination of 
$v_k$ such that there is a path from $k$ to $i_p$ in $\vec{\Gamma}$.
In the case $p=1$ this follows from the definition of $\Theta$.
For the induction step, the part that $\theta_{\alpha}(v_{i_p})$ is a
linear combination of  $v_k$ such that there is a path from $k$ to 
$i_p$ in $\vec{\Gamma}$ follows from the definition of $\Theta$. 
The part that coefficients are only $0$ or $1$ follows from the fact 
that $\Gamma$ contains no loops.  This completes the proof.
\end{proof}

\subsection{Representations by binary relations}\label{s3.3}

Consider the semigroup $\mathfrak{B}(\Gamma_0)$ of all binary 
relations on $\Gamma_0$. Fixing some bijection between $\Gamma_0$
and $\{1,2,\dots,n\}$, we may identify $\mathfrak{B}(\Gamma_0)$ with
the semigroup of all $n\times n$-matrices with coefficients $0$ or
$1$ under the natural multiplication (the usual matrix multiplication
after which all nonzero entries are treated as $1$). This identifies
$\mathfrak{B}(\Gamma_0)$ with the quotient of the semigroup 
$\mathrm{Mat}_{n\times n}(\mathbb{N}_0)$ 
(here $\mathbb{N}_0=\{0,1,2,\dots\}$) modulo the congruence
for which two matrices are equivalent if and only if they have the
same zero entries. 

As the image of the linear representation $\Theta$ 
(and also of its transpose) constructed in
Subsection~\ref{s3.2} belongs to 
$\mathrm{Mat}_{n\times n}(\mathbb{N}_0)$, composing it with the
natural projection $\mathrm{Mat}_{n\times n}(\mathbb{N}_0)\tto
\mathfrak{B}(\Gamma_0)$ we obtain a representation $\Theta'$ of 
$\mathbf{K}\mathcal{H}_{\vec{\Gamma}}$ by binary relations on $\Gamma_0$.
As matrices appearing in the image of $\Theta$ are $(0,1)$-matrices,
the representation $\Theta'$ is faithful if and only if $\Theta$ is.

\subsection{Regular actions of $\mathcal{C}_{n+1}$}\label{s3.4}

The semigroup $\mathcal{C}_{n+1}$ (which is isomorphic to the semigroup
$\mathbf{K}\mathcal{H}_{\vec{\Gamma}}$ in the case $\vec{\Gamma}$ is of the form
\eqref{eq2}) admits natural regular actions on some classical sets
of cardinality $C_{n+1}$. For example, consider the set $M_1$
consisting of all sequences $1\leq x_1\leq x_2\leq\dots\leq x_{n+1}$
of integers such that $x_i\leq i$ for all $i$ (see \cite[6.19(s)]{St1}). 
For $j=1,\dots,n$ define the action of $T_i$ (see \eqref{eq31}) on 
such a sequence as follows:
\begin{displaymath}
T_i(x_1,x_2,\dots,x_{n+1})=(x_1,\dots,x_{i-1},x_i,x_i,x_{i+2},\dots,x_{n+1}).
\end{displaymath}
It is easy to check that this indeed defines an action of $\mathcal{C}_{n+1}$
on $M_1$ by total transformations and that this action
is equivalent to the regular action of $\mathcal{C}_{n+1}$.

As another example consider the set $M_2$ of sequences of $1$'s and
$-1$'s, each appearing $n+1$ times, such that every partial sum is nonnegative 
(see \cite[6.19(r)]{St1}). For $j=1,\dots,n$ define the action of $T_i$ 
on such a sequence as follows: $T_i$ moves the $i+1$-st
occurrence of $1$ to the left and places it right after the $i$-th occurrence,
for example,
\begin{displaymath}
T_3(11-1--\mathbf{1}1--)=11-1\mathbf{1}--1--
\end{displaymath}
(here $-1$ is denoted simply by $-$ and the element which is moved is 
given in bold). It is easy to check that this 
indeed defines an action of $\mathcal{C}_{n+1}$
on $M_2$ by total transformations and that this action
is equivalent to the regular action of $\mathcal{C}_{n+1}$.

\subsection{Projective and simple linear representations}\label{s3.5}

As $\mathbf{K}\mathcal{H}_{\vec{\Gamma}}$ is a finite $\mathcal{J}$-trivial monoid,
the classical representation theory of finite semigroups
(see e.g. \cite{GMS} or \cite[Chapter~11]{GM}) applies in a straightforward 
way. Thus, from statement \eqref{thm1.3} it follows 
that $\mathbf{K}\mathcal{H}_{\vec{\Gamma}}$ has exactly $2^n$ (isomorphism classes 
of) simple modules over any field $\Bbbk$. These are constructed as 
follows: for $X\subset \Gamma_0$ the corresponding simple module 
$L_X=\Bbbk$ and for $i\in\Gamma_0$ the element $\varepsilon_i$ acts  
on $L_X$ as the identity if $i\in X$ and as zero otherwise.

The indecomposable projective cover $P_X$ of $L_X$ is combinatorial
in the sense that it is the linear span of the set 
\begin{displaymath}
\mathtt{P}_X:=\{\beta\in \mathbf{K}\mathcal{H}_{\vec{\Gamma}}:\text{ for all }
i\in\Gamma_0\text{ the equality }
\beta\varepsilon_i=\beta\text{ implies }i\in X\}
\end{displaymath}
with the action of $\mathbf{K}\mathcal{H}_{\vec{\Gamma}}$ given, 
for $\alpha\in \mathbf{K}\mathcal{H}_{\vec{\Gamma}}$ and $\beta\in \mathtt{P}_X$, by 
\begin{displaymath}
\alpha\cdot \beta=
\begin{cases}
\alpha\beta, & \alpha\beta\in \mathtt{P}_X;\\
0,&\text{ otherwise}.
\end{cases} 
\end{displaymath}

\begin{remark}\label{rem31}
{\rm  
Both Theorem~\ref{thm1}\eqref{thm1.1}-\eqref{thm1.5} 
and Subsections~\ref{s3.2}, \ref{s3.3} and \ref{s3.5} generalize mutatis
mutandis to the case of an arbitrary forest $\Gamma$ 
(the corresponding Coxeter group $W_{\Gamma}$ is infinite in general).
To prove Theorem~\ref{thm1}\eqref{thm1.1} in the general case
one should rather consider
$\mathbf{K}\mathcal{H}_{\vec{\Gamma}}$ as a quotient of $\mathbf{K}_n$
(via the epimorphism $\psi$ from Subsection~\ref{s2.4}).
}
\end{remark}

\section{Catalan numbers via enumeration of special words}\label{s4}

The above results suggest the following interpretation for
short-braid avoiding permutations.
For $n\in\mathbb{N}$ consider the alphabet $\{a_1,a_2,\dots,a_n\}$ 
and the set $\mathfrak{W}_n$ of all finite words in this alphabet.
Let $\sim$ denote the minimal equivalence relation on $\mathfrak{W}_n$
such that for any $i,j\in\{1,2,\dots,n\}$ satisfying $|i-j|>1$ and 
any $v,w\in \mathfrak{W}_n$ we have $va_ia_jw\sim va_ja_iw$. 

A word $v\in \mathfrak{W}_n$ will be called {\em strongly special}
if the following condition is satisfied: whenever
$v=v_1a_iv_2a_iv_3$ for some $i$, the word $v_2$ contains both
$a_{i+1}$ and $a_{i-1}$. In particular, both $a_1$ and $a_n$
occur at most once in any strongly special word. It is easy to
check that the equivalence class of a strongly special word consists
of strongly special words.

\begin{proposition}\label{prop32}
The number of equivalence classes of strongly special words in 
$\mathfrak{W}_n$ equals $C_{n+1}$.
\end{proposition}

\begin{proof}
We show that equivalence classes of strongly special words correspond
exactly to short-braid avoiding permutations in $S_{n+1}$.
After that the proof is completed by applying 
arguments from Subsection~\ref{s2.6}.

If $v=a_{i_1}a_{i_2}\dots a_{i_k}$ is a strongly special word, then the
corresponding permutation $s_{i_1}s_{i_2}\dots s_{i_k}\in S_{n+1}$
is obviously short-braid avoiding. 

On the other hand, any reduced expression of a short-braid avoiding
permutation corresponds to a strongly special word. Indeed, assume that 
this is not the case. Let $s_{i_1}s_{i_2}\dots s_{i_k}\in S_{n+1}$
be a reduced expression for a short-braid avoiding element 
and assume that the corresponding word
$v=a_{i_1}a_{i_2}\dots a_{i_k}$ is not strongly special. Then we may 
assume that $k$ is minimal possible, which yields that we can write
$v=a_iwa_i$ such that $w$ contains neither $a_i$ nor one of 
the elements $a_{i\pm 1}$. Without loss of generality we may assume 
that $w$ does not contain $a_{i+1}$. 

First we observe that $w$ must contain $a_{i-1}$, for otherwise 
$s_i$ would commute with all other appearing simple reflections
and hence, using $s_i^2=e$ we would obtain that our expression above
is not reduced, a contradiction. Further, we claim that $a_{i-1}$
occurs in $w$ exactly once, for $w$ does not contain $a_i$
and hence any two occurrences of $a_{i-1}$ would bound a 
proper subword of $v$ that is not strongly special, contradicting the 
minimality of $k$. 

Since $s_i$ commutes with all simple reflections appearing in 
our product but $s_{i-1}$, which, in turn, appears only once, we 
can compute that $s_ias_{i-1}bs_i=as_is_{i-1}s_ib$, which contradicts 
our assumption of short-braid avoidance.
The claim of the proposition follows.
\end{proof}

This interpretation is closely connected with $\mathbf{K}_n$.
A word $v\in \mathfrak{W}_n$ is called {\em special}
provided that the following condition is satisfied: whenever
$v=v_1a_iv_2a_iv_3$ for some $i$, then $v_2$ contains both
some $a_{j}$ with $j>i$ and some $a_{j}$ with $j<i$. In particular,
every strongly special word is special. The number of special words 
equals the cardinality of $\mathbf{K}_n$ (see \cite{KM}). 
So far there is no formula for this number.

\section{Hecke-Kiselman semigroups}\label{s5}

\subsection{Definitions}\label{s5.1}

Kiselman quotients of $0$-Hecke monoids suggest the following 
general construction. For simplicity, for every
$n\in\mathbb{N}$ we fix the set $\mathtt{N}_n:=\{1,2,\dots,n\}$
with $n$ elements. Let $\mathcal{M}_n$ denote the set of all 
simple digraphs on $\mathtt{N}_n$.
For $\Theta\in \mathcal{M}_n$ define the corresponding 
{\em Hecke-Kiselman semigroup} $\mathbf{HK}_{\Theta}$ 
(or an {\em $\mathbf{HK}$-semigroup} for short) as follows:
$\mathbf{HK}_{\Theta}$ is the monoid generated by idempotents
$e_i$, $i\in \mathtt{N}_n$, subject to the following relations
(for any $i,j\in \mathtt{N}_n$, $i\neq j$):
\begin{equation}\label{relations}
\begin{array}{|ccccc||c|}
\hline
&&\text{Relations}&&&\text{Edge between $i$ and $j$}\\
\hline\hline
e_ie_j&=&e_je_i&&&\xymatrix{i&j}\\\hline
e_ie_je_i&=&e_je_ie_j&&&\xymatrix{i\ar@/^/@{->}[r]&j\ar@/^/@{->}[l]}\\\hline
e_ie_je_i&=&e_je_ie_j&=&e_ie_j&\xymatrix{i\ar@{->}[r]&j}\\\hline
e_ie_je_i&=&e_je_ie_j&=&e_je_i&\xymatrix{i&j\ar@{->}[l]}\\
\hline
\end{array}
\end{equation}
The elements $e_1,e_2,\dots,e_n$ will be called the {\em canonical 
generators} of $\mathbf{HK}_{\Theta}$. 

\begin{example}\label{exmp1}
{\rm
\begin{enumerate}[(a)]
\item\label{ex1} If $\Theta$ has no edges, the
semigroup $\mathbf{HK}_{\Theta}$ is a {\em commutative band} isomorphic
to the semigroup $(2^{\mathtt{N}_n},\cup)$ via the map 
$e_i\mapsto \{i\}$.
\item\label{ex2} Let $\Theta\in {\mathcal{M}}_n$ 
be such that for every $i,j\in\mathtt{N}_n$, 
$i>j$, the graph $\Theta$ contains the edge $\xymatrix{i\ar@{->}[r]&j}$. 
Then the semigroup $\mathbf{HK}_{\Theta}$ coincides with the 
{\em Kiselman semigroup} $\mathrm{K}_n$ as defined in \cite{KM}. This
semigroup appeared first in \cite{Go} and was also studied in \cite{Al}.
\item\label{ex3} Let $\Gamma$ be a simply laced Dynkin diagram.
Interpret every edge of $\Gamma$ as a pair of oriented edges in
different directions and let $\Theta$ denote the corresponding simple
digraph. Then $\mathbf{HK}_{\Theta}$ is isomorphic to the $0$-Hecke monoid 
$\mathcal{H}_{\Gamma}$ as defined in Section~\ref{s1}.
\item\label{ex4}  Let $\Gamma$ be an oriented simply laced Dynkin diagram
and $\Theta$ the corresponding mixed graph.
Then $\mathbf{HK}_{\Theta}$ is isomorphic to the Kiselman
quotient $\mathbf{K}\mathcal{H}_{\Gamma}$ of the $0$-Hecke monoid 
as defined in Section~\ref{s1}.
\end{enumerate}
}
\end{example}

For $\Theta\in\mathcal{M}_n$ define the {\em opposite graph}
$\Theta^{\mathrm{op}}\in\mathcal{M}_n$ as the graph obtained from 
$\Theta\in\mathcal{M}_n$ by reversing the directions of all oriented
arrows.

\begin{proposition}\label{prop7}
For any $\Theta\in\mathcal{M}_n$, mapping $e_i$ to $e_i$ extends
uniquely to an isomorphism from $\mathbf{HK}^{\mathrm{op}}_{\Theta}$
to $\mathbf{HK}_{\Theta^{\mathrm{op}}}$. 
\end{proposition}

\begin{proof}
This follows from \eqref{relations} and the easy observation that the
two last lines of \eqref{relations} are swapped by changing the
orientation of the arrows and reading all words in the relations 
from the right to the left.
\end{proof}

\subsection{Canonical maps}\label{s5.2}

\begin{proposition}\label{prop6}
Let $\Theta,\Phi\in\mathcal{M}_n$ and assume that $\Phi$
is obtained from $\Theta$ by deleting some edges.
Then mapping $e_i$ to $e_i$ extends uni\-qu\-ely
to an epimorphism from $\mathbf{HK}_{\Theta}$ to $\mathbf{HK}_{\Phi}$.
\end{proposition}

\begin{proof}
Note that for two arbitrary idempotents $x$ and $y$ of any semigroup
the commutativity $xy=yx$ implies the braid relation
\begin{displaymath}
xyx=x(yx)=x(xy)=(xx)y=xy=x(yy)=(xy)y=(yx)y=yxy. 
\end{displaymath}
Therefore, by \eqref{relations},
in the situation as described above all relations satisfied by  
canonical generators of $\mathbf{HK}_{\Theta}$ are also satisfied by 
the corresponding canonical generators of $\mathbf{HK}_{\Phi}$. This 
implies that mapping $e_i$ to $e_i$ extends uniquely to an homomorphism 
from $\mathbf{HK}_{\Theta}$ to $\mathbf{HK}_{\Phi}$. This homomorphism
is surjective as its image contains all generators of 
$\mathbf{HK}_{\Phi}$.
\end{proof}

We call the epimorphism constructed in Proposition~\ref{prop6} 
the {\em canonical projection} and denote it by 
$\mathfrak{p}_{\Theta,\Phi}$. 

For $\Theta$ and $\Phi$ as above we will write $\Theta\geq\Phi$. Then
$\geq$ is a partial order on $\mathcal{M}_n$ and it defines on 
$\mathcal{M}_n$ the structure of a distributive lattice. 
The maximum element of ${\mathcal{M}}_n$ is the full unoriented 
graph on $\mathtt{N}_n$, which we denote by $\mathbf{max}$. The minimum 
element of ${\mathcal{M}}_n$ is the empty 
graph (the graph with no edges), which we denote by $\mathbf{min}$. 
By Example~\ref{exmp1}\eqref{ex1}, the semigroup $\mathbf{HK}_{\mathbf{min}}$
is a commutative band isomorphic to $(2^{\mathtt{N}_n},\cup)$.
Further, for any $\Theta\in{\mathcal{M}}_n$ 
we have the canonical projections $\mathfrak{p}_{\mathbf{max},\Theta}:
\mathbf{HK}_{\mathbf{max}}\tto \mathbf{HK}_{\Theta}$ and
$\mathfrak{p}_{\Theta,\mathbf{min}}:
\mathbf{HK}_{\Theta}\tto \mathbf{HK}_{\mathbf{min}}$.

For $w\in \mathbf{HK}_{\Theta}$ we define the {\em content} of
$w$ as $\mathfrak{c}(w):=\mathfrak{p}_{\Theta,\mathbf{min}}(w)$. 
This should be understood as the set of canonical generators 
of $\mathbf{HK}_{\Theta}$ appearing in any decomposition of $w$ 
into a product of canonical generators. Under the identification of
$\mathbf{HK}_{\mathbf{min}}$ and $(2^{\mathtt{N}_n},\cup)$, by
$|\mathfrak{c}(w)|$ we understand the number of generators used
to obtain $w$. In particular, $|\mathfrak{c}(e)|=0$ and
$|\mathfrak{c}(e_i)|=1$ for all $i$.

Let $m,n\in\mathbb{N}$, $\Theta\in\mathcal{M}_m$ and
$\Phi\in\mathcal{M}_n$. Assume that $f:\Theta\to \Phi$ is
a full embedding of graphs, meaning that it is an injection on
vertexes and edges and its image in $\Phi$ is a full subgraph
of $\Phi$. 

\begin{proposition}\label{prop8}
In the situation above mapping $e_i$ to $e_{f(i)}$ induces a
monomorphism from $\mathbf{HK}_{\Theta}$
to $\mathbf{HK}_{\Phi}$.
\end{proposition}

\begin{proof}
From \eqref{relations} and our assumptions on $f$ it follows 
that $e_{f(i)}$'s satisfy all the corresponding defining relations 
satisfied by $e_i$'s. This implies that mapping $e_i$ to $e_{f(i)}$ 
induces a homomorphism $\varphi$ from $\mathbf{HK}_{\Theta}$
to $\mathbf{HK}_{\Phi}$. 

To prove that this homomorphism is injective it is enough to
construct a left inverse. Similarly to the previous paragraph,
from \eqref{relations} and our assumptions on  $f$ it follows that 
mapping $e_{f(i)}$ to $e_i$ and all other canonical generators of
$\mathbf{HK}_{\Phi}$ to $e$ induces a homomorphism $\psi$ from
$\mathbf{HK}_{\Phi}$ to $\mathbf{HK}_{\Theta}$. 
It is straightforward to verify that $\psi\circ\varphi$ acts
as the identity on all generators of $\mathbf{HK}_{\Theta}$.
Therefore $\psi\circ\varphi$ coincides with the identity.
The injectivity of $\varphi$ follows.
\end{proof}

We call the monomorphism constructed in Proposition~\ref{prop8} 
the {\em canonical injection} and denote it by 
$\mathfrak{i}_{f}$. 

\subsection{Classification up to isomorphism}\label{s5.3}

The main result of this subsection is the following classification of
Hecke-Kiselman semigroups up to isomorphism in terms of the underlying 
mixed graphs.

\begin{theorem}\label{thm4}
Let $m,n\in\mathbb{N}$, $\Theta\in\mathcal{M}_m$ and
$\Phi\in\mathcal{M}_n$. Then the semigroups $\mathbf{HK}_{\Theta}$
and $\mathbf{HK}_{\Phi}$ are isomorphic if and only if the graphs
$\Theta$ and $\Phi$ are isomorphic. In particular,
if $\mathbf{HK}_{\Theta}$ and $\mathbf{HK}_{\Phi}$ 
are isomorphic, then $m=n$.
\end{theorem}

\begin{proof}
Let $f:\Theta\to\Phi$ be an isomorphism of graphs with
inverse $g$. By Proposition~\ref{prop8} we have the corresponding
natural injections $\mathfrak{i}_{f}:\mathbf{HK}_{\Theta}\to
\mathbf{HK}_{\Phi}$ and $\mathfrak{i}_{g}:\mathbf{HK}_{\Phi}\to
\mathbf{HK}_{\Theta}$. By definition, both 
$\mathfrak{i}_{g}\circ\mathfrak{i}_{f}$ and $\mathfrak{i}_{f}\circ\mathfrak{i}_{g}$ act as identities on 
the generators of $\mathbf{HK}_{\Theta}$ and $\mathbf{HK}_{\Phi}$,
respectively. Hence $\mathfrak{i}_{f}$ and $\mathfrak{i}_{g}$ are 
mutually inverse isomorphisms. This proves the ``if'' part of the 
first claim of the theorem.

\begin{lemma}\label{lem5}
We have
\begin{displaymath}
\mathrm{Irr}(\mathbf{HK}_{\Phi})=\{e_1,e_2,\dots,e_n\}=
\{w\in \mathbf{HK}_{\Phi}:|\mathfrak{c}(w)|= 1\}. 
\end{displaymath}
\end{lemma}

\begin{proof}
From the definitions we see that $\mathrm{Irr}(\mathbf{HK}_{\Phi})$
is contained in any generating system for $\mathbf{HK}_{\Phi}$, in particular,
in $\{w\in \mathbf{HK}_{\Phi}:|\mathfrak{c}(w)|= 1\}$.

Since all canonical generators of $\mathbf{HK}_{\Phi}$ are idempotents,
it follows that $\{w\in \mathbf{HK}_{\Phi}:|\mathfrak{c}(w)|\leq 1\}
\subset \{e,e_1,e_2,\dots,e_n\}$.
It is straightforward to verify that $\{e_1,e_2,\dots,e_n\}\subset 
\mathrm{Irr}(\mathbf{HK}_{\Phi})$, which completes the proof.
\end{proof}

Assume that $\varphi:\mathbf{HK}_{\Theta}\to\mathbf{HK}_{\Phi}$
is an isomorphism. Then $\varphi$ induces a bijection from 
$\mathrm{Irr}(\mathbf{HK}_{\Theta})$ to
$\mathrm{Irr}(\mathbf{HK}_{\Phi})$, which implies 
$m=n$ by comparing the cardinalities of these sets
(see Lemma~\ref{lem5}). 
This proves the second claim of the theorem.

Let $e_i$ and $e_j$ be two different canonical generators of 
$\mathbf{HK}_{\Theta}$. By \eqref{relations}, in the case when 
the graph $\Theta$ contains no edge between $i$ and $j$ the 
elements $e_i$ and $e_j$ commute in $\mathbf{HK}_{\Theta}$. 
As $\varphi$ is an isomorphism, we get that $\varphi(e_i)=e_s$ 
and $\varphi(e_j)=e_t$ commute in $\mathbf{HK}_{\Phi}$.
Using \eqref{relations} again we obtain that the graph $\Phi$ 
contains no edge between $s$ and $t$. 

Similarly, comparing the subsemigroup of $\mathbf{HK}_{\Theta}$ generated 
by $e_i$ and $e_j$ with the subsemigroup of $\mathbf{HK}_{\Phi}$ 
generated by $\varphi(e_i)$ and $\varphi(e_j)$ for all other possibilities 
for edges between $i$ and $j$, we obtain that $\varphi$ induces a graph isomorphism from $\Theta$ to $\Phi$. This proves the ``only if'' part 
of the first claim of the theorem and thus completes the proof.
\end{proof}

\begin{corollary}\label{cor10}
For $\Phi\in \mathcal{M}_n$ the set $\{e_1,e_2,\dots,e_n\}$ is the
unique irreducible generating system of $\mathbf{HK}_{\Phi}$.
\end{corollary}

\begin{proof}
That $\{e_1,e_2,\dots,e_n\}$ is an irreducible generating system 
of $\mathbf{HK}_{\Phi}$ follows from the definitions. On the other
hand, that any generating system of $\mathbf{HK}_{\Phi}$ contains
$\{e_1,e_2,\dots,e_n\}$ follows from the proof of Lemma~\ref{lem5}.
This implies the claim.
\end{proof}

From the above it follows that the number of isomorphism classes
of semigroups $\mathbf{HK}_{\Theta}$, $\Theta\in\mathcal{M}_n$,
equals the number of simple digraphs. The latter is known as the
sequence A000273 of the On-Line Encyclopedia of Integer Sequences.

\subsection{Some open problems}\label{s5.4}

Here is a short list of some natural questions on Hecke-Kiselman 
semigroups:
\begin{itemize}
\item For which $\Theta$ is $\mathbf{HK}_{\Theta}$ finite?  
\item For which $\Theta$ is $\mathbf{HK}_{\Theta}$ $\mathcal{J}$-trivial?  
\item For a fixed $\Theta$, what is the smallest $n$ for which 
there is a faithful representation of $\mathbf{HK}_{\Theta}$ by
$n\times n$ matrices (over $\mathbb{Z}$ or $\mathbb{C}$)?  
\item For a fixed $\Theta$, how to construct a faithful representation
of $\mathbf{HK}_{\Theta}$ by (partial) transformations?  
\item What is a canonical form for an element of $\mathbf{HK}_{\Theta}$?
\end{itemize}

\vspace{0.2cm}

\noindent
O.G.: Department of Mechanics and Mathematics, Kyiv Taras Shev\-chen\-ko
University, 64, Volodymyrska st., UA-01033, Kyiv, UKRAINE,\\
e-mail: {\tt ganiyshk\symbol{64}univ.kiev.ua}
\vspace{0.2cm}

\noindent
V.M.: Department of Mathematics, Uppsala University, Box 480,
SE-75106, Uppsala, SWEDEN, e-mail: {\tt mazor\symbol{64}math.uu.se},\\
web: ``http://www.math.uu.se/$\tilde{\hspace{2mm}}$mazor''
\vspace{0.5cm}


\begin{thebibliography}{99999}
\bibitem[Al]{Al} S.~Alsaody, Determining the elements of a semigroup,
U.U.D.M. Report {\bf 2007:3}, Uppsala University.
\bibitem[BJS]{BJS} S.~Billey, W.~Jockusch, R.~Stanley; Some 
combinatorial properties of Schubert polynomials.  J. Algebraic 
Combin. {\bf 2}  (1993),  no. 4, 345--374. 
\bibitem[BB]{BB} A.~Bj{\"o}rner,  F.~Brenti; Combinatorics of Coxeter 
groups. Graduate Texts in Mathematics, {\bf 231}. Springer, New York, 2005.
\bibitem[Ca]{Ca} R.~Carter; Representation theory of the $0$-Hecke 
algebra.  J. Algebra  {\bf 104}  (1986),  no. 1, 89--103.
\bibitem[Do]{Do} J.~Doyen; {\'E}quipotence et unicit{\'e} de 
syst{\`e}mes g{\'e}n{\'e}rateurs minimaux dans certains mono{\"\i}des. 
Semigroup Forum {\bf 28} (1984), no. 1-3, 341--346.  
\bibitem[Fa]{Fa} M.~Fayers; $0$-Hecke algebras of finite Coxeter groups.  
J. Pure Appl. Algebra  {\bf 199}  (2005),  no. 1-3, 27--41.
\bibitem[FG]{FG} S.~Fomin, C.~Greene; Noncommutative Schur functions and 
their applications. Discrete Math.  {\bf 193}  (1998),  no. 1-3, 179--200.
\bibitem[GM1]{GM0} O.~Ganyushkin, V.~Mazorchuk; On the structure 
of $\mathcal{IO}_n$.  Semigroup Forum  {\bf 66}  (2003),  no. 3, 455--483.
\bibitem[GM2]{GM01} O.~Ganyushkin, V.~Mazorchuk; Combinatorics of 
nilpotents in symmetric inverse semigroups. Ann. Comb. {\bf 8}  
(2004),  no. 2, 161--175.
\bibitem[GM3]{GM} O.~Ganyushkin, V.~Mazorchuk; Classical finite 
transformation semigroups. An introduction. Algebra and Applications, 
{\bf 9}. Springer-Verlag London, Ltd., London,  2009.
\bibitem[GMS]{GMS} O.~Ganyushkin, V.~Mazorchuk, B.~Steinberg; 
On the irreducible representations of a finite semigroup. 
Proc. Amer. Math. Soc. {\bf 137} (2009), no. 11, 3585--3592.
\bibitem[Go]{Go} R.~Golovko; On some properties of Kiselman's
semigroup, 4-th international algebraic conference in Ukraine, Lviv,
August 4-9, 2003, Collection of abstracts, 81-82. 
\bibitem[Hi]{Hi} P.~Higgins; Combinatorial aspects of semigroups 
of order-preserving and decreasing functions.  Semigroups 
(Luino, 1992),  103--110, World Sci. Publ., River Edge, NJ, 1993. 
\bibitem[HNT]{HNT} F.~Hivert, J.-C.~Novelli, J.-Y.~Thibon; Yang-Baxter 
bases of $0$-Hecke algebras and representation theory of 
$0$-Ariki-Koike-Shoji algebras. Adv. Math. {\bf 205} (2006), no. 2, 504--548.
\bibitem[HST1]{HST} F.~Hivert, A.~Schilling, N.~Thi{\'e}ry; Hecke group 
algebras as quotients of affine Hecke algebras at level 0. J. Combin. 
Theory Ser. A {\bf 116} (2009), no. 4, 844--863.
\bibitem[HST2]{HST2} F.~Hivert, A.~Schilling, N.~Thi{\'e}ry;
The biHecke monoid of a finite Coxeter group, Preprint arXiv:0912.2212.
\bibitem[HT]{HT} F.~Hivert, N.~Thi{\'e}ry; Representation theories of some towers of algebras related to the symmetric groups and their Hecke algebras.
Preprint arXiv:math/0607391.
\bibitem[KM]{KM} G.~Kudryavtseva, V.~Mazorchuk; On Kiselman's 
semigroup.  Yokohama Math. J.  {\bf 55}  (2009),  no. 1, 21--46.
\bibitem[Ll]{Ll} G.~Lallement; Semigroups and combinatorial 
applications. Pure and Applied Mathematics. A Wiley-Interscience 
Publication. John Wiley \& Sons, New York-Chichester-Brisbane, 1979.
\bibitem[Ma]{Ma} A.~Mathas; Iwahori-Hecke algebras and Schur algebras of 
the symmetric group. University Lecture Series, {\bf 15}. 
American Mathematical Society, Providence, RI, 1999.
\bibitem[McN]{McN} P.~McNamara; EL-labelings, supersolvability and 
$0$-Hecke algebra actions on posets.  J. Combin. Theory Ser. A  
{\bf 101 } (2003),  no. 1, 69--89. 
\bibitem[No]{No} P.~Norton; $0$-Hecke algebras.  J. Austral. Math. 
Soc. Ser. A  {\bf 27}  (1979), no. 3, 337--357. 
\bibitem[NT1]{NT} J.-C.~Novelli, J.-Y.~Thibon; Noncommutative symmetric 
Bessel functions.  Canad. Math. Bull.  {\bf 51}  (2008),  no. 3, 424--438. 
\bibitem[NT2]{NT2} J.-C.~Novelli, J.-Y.~Thibon; Noncommutative 
symmetric functions and Lagrange inversion.  Adv. in Appl. Math.  
{\bf 40} (2008),  no. 1, 8--35.
\bibitem[Pi]{Pi}  J.-{\'E}.~Pin; Mathematical Foundations of Automata 
Theory. Preprint available at www.liafa.jussieu.fr/~jep/PDF/MPRI/MPRI.pdf
\bibitem[So]{So} A.~Solomon; Catalan monoids, monoids of local 
endomorphisms, and their presentations. Semigroup Forum {\bf 53} 
(1996), no. 3, 351--368.
\bibitem[St]{St1} R.~Stanley; Enumerative combinatorics. Vol. 2. 
Cambridge Studies in Advanced Mathematics, {\bf 62}. 
Cambridge University Press, Cambridge, 1999.
\end{thebibliography}
\end{document}